\documentclass[a4paper, 12pt]{article}
\usepackage{epsfig}
\usepackage[latin1]{inputenc}
\usepackage{amsmath,amssymb,amsthm,amsxtra}
\usepackage{enumerate}
\usepackage{verbatim}
\usepackage{amsfonts}
\usepackage{relsize}

\usepackage{longtable}

\usepackage{tikz}
\usetikzlibrary{matrix}

\usepackage{hyperref}

\usepackage{enumitem}
\usepackage{cite}
\usepackage{graphicx}
\usepackage{amssymb}
\usepackage{amsopn}
\usepackage{multirow}
\usepackage{multicol}
\usepackage{epstopdf}
\usepackage{setspace}
\usepackage{cite}
\usepackage{geometry}
\newtheorem{thm}{Theorem}[section]

\newtheorem{lem}[thm]{Lemma}
\newtheorem{prop}[thm]{Proposition}
\newtheorem{defn}[thm]{Definition}
\newtheorem{deflem}[thm]{Definition/Lemma}

\newcommand{\RR}{\mathbb{R}}
\newcommand{\CC}{\mathbb{C}}
\newcommand{\ZZ}{\mathbb{Z}}
\newcommand{\QQ}{\mathbb{Q}}

\newcommand{\PP}{\mathbb{P}}

\newcommand{\HH}{\mathbb{H}}

\newcommand{\cc}{\mathcal{C}}
\newcommand{\dc}{\mathcal{D}}
\newcommand{\ec}{\mathcal{E}}
\newcommand{\oc}{\mathcal{O}}
\newcommand{\lc}{\mathcal{L}}

\newcommand{\ttt}{\texttt}

\newcommand{\la}{\langle}
\newcommand{\ra}{\rangle}
\newcommand{\bsm}{\left ( \begin{smallmatrix} }
\newcommand{\esm}{\end{smallmatrix} \right )}
\newcommand{\bma}{\begin{pmatrix}}
\newcommand{\ema}{\end{pmatrix}}
\newcommand{\opn}{\operatorname}
\newcommand{\op}{\oplus}
\newcommand{\fc}{\mathcal{F}}
\newcommand{\bi}{\begin{itemize}}
\newcommand{\ei}{\end{itemize}}
\newcommand{\bpm}{\begin{pmatrix}}
\newcommand{\epm}{\end{pmatrix}}
\title{General type results for moduli of deformation generalised Kummer varieties}
\author{Matthew Dawes}
\date{}

\begin{document}
\maketitle
\abstract{In \cite{DawesKod}, families of orthogonal modular varieties $\fc(\Gamma)$ associated with moduli spaces of compact hyperk\"ahler manifolds of deformation generalised Kummer type (also known as \emph{`deformation generalised Kummer varieties'}) were studied.
  The orthogonal modular varieties were defined for an integer $2d$, corresponding to the degree of polarisation of the associated hyperk\"ahler manifolds.
  It was shown in \cite{DawesKod} that the modular varieties are of general type when $2d$ is square-free and sufficiently large. 
The purpose of this paper is to show that the square-free condition can be removed.  
}
\section{Introduction}
Let $L$ be a lattice of signature $(2,n)$ with orthogonal group $\opn{O}(L)$ and suppose $\dc_L$ is the component of the Hermitian symmetric space
\[
\Omega_L = \{ [x] \in \PP(L \otimes \CC) \mid (x,x)=0, (x, \overline{x})>0 \}
\]
preserved by the kernel $\opn{O}^+(L \otimes \RR)$ of the real spinor norm on $\opn{O}(L \otimes \RR)$ \cite{kneser}.
An \emph{orthogonal modular variety} $\fc(\Gamma)$ is a quotient of the form 
\[
\fc(\Gamma) = \dc_L / \Gamma  
\]
where $\Gamma \subset \opn{O}(L) \cap \opn{O}^+(L \otimes \RR)$ is a subgroup of finite index.
A classical problem in algebraic geometry is to determine the Kodaira dimension of families of orthogonal modular varieties.

In \cite{DawesKod}, families of orthogonal modular varieties related to moduli spaces of compact hyperk\"ahler manifolds of deformation generalised Kummer type were studied.
They were defined for the lattice
\begin{equation}\label{GKlattice}
L = U \op U \op \la -6 \ra \op \la - 2d \ra,
\end{equation}
where $d \geq 48$.
In \cite[Theorem 11.3]{DawesKod}, it was shown that $\fc(\Gamma)$ is of general type when $2d$ is square-free and sufficiently large.
The argument of \cite{DawesKod} involves constructing pluricanonical forms on a smooth projective model of $\fc(\Gamma)$ using modular forms for $\Gamma$.
Not all modular forms define pluricanonical forms: in particular, obstructions arise from certain types of singularities in $\fc(\Gamma)$ and compactifications.
In \cite{DawesKod}, a square-free condition was imposed on $2d$ to eliminate some of these singularities.
Here we show that the square-free condition can be removed when $d \gg 0$.
\subsection{Notation and preliminaries}
\subsubsection{Lattices and orthogonal groups}
Standard references for lattices are \cite{SPLAG} and \cite{Nikulin}.
We let $U$ denote the hyperbolic plane and use $\la 2d \ra$ to denote the rank 1 lattice generated by a single element $v$ of squared length $v^2 = (v,v)=2d$.
We let $L_1 \op L_2$ denote the orthogonal direct sum of lattices $L_1$ and $L_2$ and let $L^{\vee}$ denote the dual lattice $L^{\vee}:=\opn{Hom}(L, \ZZ)$ of $L$.
The \emph{stable orthogonal group} $\widetilde{\opn{O}}(L)$  \cite{Handbook} is defined as the subgroup of $\opn{O}(L)$ acting trivially on the \emph{discriminant group} $D(L):=L^{\vee}/L$ of $L$ \cite{Nikulin}.
We let $\widetilde{\opn{SO}}^+(L):=\widetilde{\opn{O}}^+(L) \cap \opn{SO}(L)$
and use a superscript $+$ to denote the intersection of a subgroup of $\opn{O}(L)$ with $\opn{O}^+(L \otimes \RR)$ (e.g. $\widetilde{\opn{O}}^+(L)$).
\subsubsection{Singularities and moduli of deformation generalised Kummer varieties}
From now on, we assume the modular varieties $\fc(\Gamma)$ are those defined in \cite{DawesKod} for $d \geq 48$.
For $L$ as in \eqref{GKlattice}, the group $\Gamma$ is given by \cite[Proposition 2.1]{DawesKod} 
\begin{equation}\label{GammaDef}
\Gamma = \{ g \in \opn{O}^+(L) \mid gv^* \equiv v^* \bmod{L} \text{ and } gw^* \equiv \opn{det}(g) w^* \bmod{L} \},
\end{equation}
where $v^* = -\frac{-1}{2d} v$ and $w^* = \frac{-1}{6}w$ for the generators $v$ and $w$ of the $\la - 2d \ra$ and $\la -6 \ra$ factors of $L$, respectively.
We let $\overline{\fc}(\Gamma)$ denote a toroidal compactification of $\fc(\Gamma)$ \cite{AMRT}.
In a neighbourhood of a point $P$, $\overline{\fc}(\Gamma)$ is locally isomorphic to $\CC^4/G_P$ \cite{Cartan} for a finite group $G_P \subset \opn{GL}(4, \CC)$.
For a subgroup $G \subset G_P $ or $g \in G_P$, we let  $X_G := \CC^4/G$ and $X_g:=\CC^n/\la g \ra$.
Depending on the context, we may use $\frac{1}{m}(a_1, \ldots, a_n)$ denote either the matrix $\opn{diag}(\xi^{a_1}, \ldots, \xi^{a_n})$ or the quotient $\CC^n/\la \opn{diag}(\xi^{a_1}, \ldots, \xi^{a_n}) \ra$ for $\xi:=e^{2 \pi i / m}$.

\begin{defn}
  Let $P \in \overline{\fc}(\Gamma)$ and suppose $I \neq g \in \opn{GL}(4, \CC)$. 
  If there exists $h \in G_P$ such that $h \sim g$, we call $P$ a \textbf{$g$-point}; if $g$ is also a reflection, we call $P$ a \textbf{$\sigma$-point}.
  If $P$ is a $g$-point and $X_g$ is a non-canonical singularity \cite{YPG}, we call $P$ a \textbf{non-canonical $g$-point}. 
\end{defn}
\subsubsection{Modular forms and the low-weight cusp form trick}\label{LWCFT_sec}
\begin{defn}[{\hspace{1sp}\cite[p.495]{Handbook}}]
Let $\dc_L^{\bullet}$ denote the affine cone of $\dc_L$.
  A \textbf{modular form} $f$ of weight $k$ and character $\chi$ for $\Gamma$ is a holomorphic function $f: \dc_L^{\bullet} \rightarrow \CC$ such that $f(tZ)=t^{-k}f(Z)$ and $f(gZ) = \chi(g) f(Z)$ for all $Z \in \dc_L^{\bullet}$, all $t \in \CC^*$ and all $g \in \Gamma$.
We let $M_k(\Gamma)$ denote the space of weight $k$-modular forms for $\Gamma$ and let $S_k(\Gamma) \subset M_k(\Gamma)$ denote the subspace of cusp forms.
\end{defn}
We will obtain general type results for $\fc(\Gamma)$ by constructing pluricanonical forms $\Omega(f)$ on a smooth projective model $\widetilde{\fc}(\Gamma)$ of $\overline{\fc}(\Gamma)$ using the \emph{`low-weight cusp form trick'} \cite{GritsenkoSankaranAbelian, GritsenkoHulekIrrational, GritsenkoAbelianK3, GHSK3}.
Suppose
\begin{equation}\label{mfdef}
  f=f_a^k h \in M_{4k}(\Gamma)
\end{equation}
where $0 \neq f_a \in S_a(\Gamma)$, $a<4$ and $h \in M_{(4-a)}(\Gamma)$.
If $dZ$ is a holomorphic volume form for $\dc_L$ then  $\Omega(f):=f (dZ)^{\otimes k}$ is $\Gamma$-invariant on $\dc_L$ and defines a pluricanonical form on $\widetilde{\fc}(\Gamma)$ away from the $\sigma$-points and non-canonical $g$-points of $\overline{\fc}(\Gamma)$ \cite{AMRT, MaIrregularCusps, Handbook, GHSK3}.
\begin{defn}
  If $P \in \overline{\fc}(\Gamma)$ is a $g$-point, we say that $\Omega(f)$ is \textbf{$g$-extensible at $P$} if $\Omega(f)$ extends to a desingularisation of $X_g$, or to $X_{\sigma}$  if $P$ is a $\sigma$-point.
  If $\Omega(f)$ is $g$-extensible at all $g$-points of $\overline{\fc}(\Gamma)$, we say that $\Omega(f)$ is \textbf{$g$-extensible everywhere}.
  If $\Omega(f)$ is $g$-extensible at all $g$-points of a subset $S \subset \overline{\fc}(\Gamma)$, we say that $\Omega(f)$ is \textbf{$g$-extensible on $S$}. 
\end{defn}
In \cite{DawesKod}, it was shown that the non-canonical $g$-points of $\overline{\fc}(\Gamma)$ occur for $g=\frac{1}{6}(0,1,1,2)$ and $\frac{1}{6}(1,1,1,2)$ and that if $\Omega(f)$ is $\sigma$ and $g$-extensible everywhere then $\Omega(f)$ defines a pluricanonical form on $\widetilde{\fc}(\Gamma)$ \cite[Proposition 3.1]{Tai} \cite{GHSK3, DawesKod}.
\sloppy Obstruction spaces
\begin{align*}
  \opn{RefObs}_{(4-a)k}(\Gamma) \subset M_{(4-a)k}(\Gamma)&& \text{and} &&  \opn{EllInt}_{(4-a)k}(\Gamma, \ec) \subset M_{(4-a)k}(\Gamma)
\end{align*}
were defined \cite[(15), (16), Lemma 10.10]{DawesKod} for a set of lattice embeddings $\ec$.
It was shown that $\Omega(f)$ is $\sigma$-extensible everywhere if $h$ lies outside of $\opn{RefObs}_{(4-a)k}(\Gamma)$ \cite[Theorem 3.2]{DawesKod} \cite{HMvol};     
and  if $h$ lies outside of $\opn{EllInt}_{(4-a)k}(\Gamma, \ec)$ (called $\opn{EllObs}_{(4-a)k}(\Gamma, \ec)$ in \cite{DawesKod}) then $\Omega(f)$ is $\frac{1}{6}(0,1,1,2)$-extensible everywhere and $\frac{1}{6}(1,1,1,2)$-extensible on $\fc(\Gamma)$  \cite[Proposition 6.2/Proposition 6.3]{DawesKod}.
The situation at the boundary is more complicated:
for certain values of $d$, $\frac{1}{6}(1,1,1,2)$-points may occur in the boundary (but never if $2d$ is square-free, as assumed in \cite{DawesKod}).
Singularities lying along curves in the interior of $\overline{\fc}(\Gamma)$ may also run out to the boundary, resulting in $\frac{1}{6}(0,1,1,2)$-points (Lemma \ref{0112-ext-lem}). 
In Lemma \ref{0112-ext-lem}, we show that $\Omega(f)$ is $\frac{1}{6}(0,1,1,2)$-extensible everywhere by modifying the arguments of \cite{DawesKod}, but a new approach is required for the $\frac{1}{6}(1,1,1,2)$-points.
In Definition/Lemma \ref{EllCuspDefLem} we introduce a new obstruction space 
\[
\opn{EllCusp}_{(4-a)k}(\Gamma) \subset M_{(4-a)k}(\Gamma)
\]
and show that $\Omega(f)$ is $\frac{1}{6}(1,1,1,2)$-extensible everywhere if $h$ lies outside of $\opn{EllCusp}_{(4-a)k}(\Gamma)$, $\opn{RefObs}_{(4-a)k}(\Gamma)$ and $\opn{EllInt}_{(4-a)k}(\Gamma, \ec)$.
As in \cite{DawesKod}, we can then obtain general type results by comparing the growth of the obstruction spaces with $M_{(4-a)k}(\Gamma)$.
Non-trivial cusp forms $f_3$ exist for $d \gg 0$\cite[Corollary 7.1]{DawesKod} and, by Hirzebruch-Mumford proportionality \cite{HMvol}, there exist  $\alpha_0(a)$, $\alpha_1(a)$, $\alpha_2(a)$,  $\alpha_3(a)$ such that
\begin{align*}
  &\opn{dim} M_{k}(\Gamma) = \alpha_0(a)k^4 + O(k^3),
  && 
\opn{dim} \opn{RefObs}_{k}(\Gamma) =  \alpha_1(a)k^4 + O(k^3), \\
&  \opn{dim} \opn{EllInt}_{k}(\Gamma, \ec) =  \alpha_2(a)k^4 + O(k^3),
&&
  \opn{dim} \opn{EllCusp}_{k}(\Gamma) = \alpha_3(a)k^4 + O(k^3).
\end{align*}
Therefore, $\fc(\Gamma)$ is of general type if 
\begin{equation}\label{alphabound}
\alpha_0(a) - \alpha_1(a) - \alpha_2(a) - \alpha_3(a) > 0,
\end{equation}
which we show is always satisfied for $d \gg 0$ (Theorem \ref{gtthm}).
 \section{Singularities in the boundary}\label{sings}
The obstruction space $\opn{EllCusp}_{k}(\Gamma)$ takes a contribution from each of the $\frac{1}{6}(1,1,1,2)$-points in the boundary of $\overline{\fc}(\Gamma)$: we bound the number of such points in Proposition \ref{Pcount_prop}.
Boundary components of $\overline{\fc}(\Gamma)$ lie above boundary points and curves of $\fc(\Gamma)$ \cite{AMRT, BorelJi, Piatetski-Shapiro, Satake1, Handbook}, which are in bijection with $\Gamma$-orbits of primitive totally isotropic sublattices of rank 1 and rank 2 in $L$, respectively \cite[Theorem 5.11]{Handbook}.
By \cite[Theorem 2.17]{GHSK3}, we will assume that all $\frac{1}{6}(1,1,1,2)$-points in the boundary of $\overline{\fc}(\Gamma)$ lie above boundary curves.
\subsubsection{Local coordinates for toroidal compactifications}\label{toroidal_coords}
Let $E \subset L$ be a primitive totally isotropic sublattice of rank 2.
As explained in \cite[Lemma 2.24]{GHSK3}, any $\ZZ$-basis  $\{ e_1, e_2 \}$ of $E$ can be extended to a $\QQ$-basis for $L$ with Gram matrix
\begin{equation}\label{Gram0}
   Q_E:=
  ((e_i, e_j)) =
\bpm
0 & 0 & A \\
0 & B & 0 \\
A & 0 & 0
\epm,
\end{equation}
where $B$ is an integral binary quadratic form and $A = \opn{diag}(a,ae)$ for $a,e \in \ZZ$.
\begin{defn}
  If $E \subset L$ is a primitive totally isotropic sublattice of rank 2, we say that $E$ (or the associated boundary curve) is of \emph{type $(a,e)$} if it defines $a$ and $e$ as in \eqref{Gram0}.
\end{defn}
For $E$ as above, fix the basis \eqref{Gram0} and let $\cc$ be the associated boundary curve.
As in \cite{Kondo, GHSK3}, we describe $\overline{\fc}(\Gamma)$ in a neighbourhood of $\cc$.
If $N(E)$ is the stabiliser of $E$ in $\opn{O}(L \otimes \RR)$, 
let $W(E)$ denote the unipotent radical of $N(E)$ and 
let $U(E)$ denote the centre of $W(E)$.
We use a subscript $\ZZ$ to denote an intersection with $\Gamma$ (e.g. $N(E)_{\ZZ}$). 
Let $\dc_L(E):=U(E) \dc_L$, taken in the compact dual of $\dc_L$.
By taking homogeneous coordinates $[t_1:\ldots:t_6]$ for $\PP(L\otimes \CC)$ we have, 
\begin{align*}
  \dc_L(E)
  & \cong U(E)_{\CC} \times \CC^2 \times \cc \\
  & = \{(z, \underline{w}, \tau) \in \CC \times \CC^2 \times \HH^+ \} \\
  & = \{[z:t_2:\underline{w}:\tau:1] \in \PP(L \otimes \CC) \mid 2aet_2 = -2az \tau - \underline{w} B^{\intercal} \underline{w} \}.
\end{align*}
On the basis \eqref{Gram0}, the group $N(E)$ is given by \cite[Lemma 2.25]{GHSK3} 
\begin{equation}\label{nc}
   N(E)  =
  \left
  \{
  \begin{pmatrix}
    U & V & W \\
    0 & X & Y \\
    0 & 0 & Z
  \end{pmatrix}
  \mbox{\larger[30] $\mid$ }                                                                
  \begin{matrix}
    {}^{\intercal} UAZ=A, {}^{\intercal}XBX=B, {}^{\intercal}XBY+{}^{\intercal}VAZ=0,\\
    {}^{\intercal} YBY + {}^{\intercal} ZAW + {}^{\intercal} WAZ=0,\
    \opn{det} U>0\end{matrix}\right\}
\end{equation}
and the action of $g \in N(E)$ on $\dc_L(E)$ is given by
\begin{equation}\label{gaction}
  \begin{cases}
  z  \mapsto 
  \frac{z}{\det Z} + (c \tau + d)^{-1} \left( \frac{c}{2 a \det Z}{}^{\intercal} \underline{w} B \underline{w}
  +\underline{V}_1\underline{w} + W_{11} \tau + W_{12} \right) \\
  \underline{w}  \mapsto  (c\tau+d)^{-1}
  ( X \underline{w} + Y 
  \bsm
    \tau \\ 1
    \esm
    )\\
    \tau  \mapsto  (a \tau + b)/(c \tau + d),
  \end{cases}
\end{equation}
where
\[
Z = \bpm a & b \\ c & d \epm
\]
and ${\underline V}_i$ is the $i$-th row of the matrix $V$ in \eqref{nc} \cite[Proposition 2.26]{GHSK3}\cite[p. 259]{Kondo}.
By setting $u:=\opn{exp}_e(z):=e^{2 \pi i z/e}$, we obtain a coordinate for the torus $T(E):=U(E)_{\CC}/U(\cc)_{\ZZ}$ in $\dc_L(E)/U(E)_{\ZZ}$.
In a neighbourhood of $\cc$, the toroidal compactification $\overline{\fc}(\Gamma)$ is obtained allowing $u=0$ and taking the resulting quotient by $G(E):=N(E)_{\ZZ}/U(E)_{\ZZ}$ \cite{GHSK3}.
By \cite[Lemma 4.11]{DawesKod}, if the boundary component corresponding to $E$ contains $\frac{1}{6}(0,1,1,2)$ or $\frac{1}{6}(1,1,1,2)$-points, then a $\ZZ$-basis for $E$ can be extended to a $\ZZ$-basis for $L$ with Gram matrix 
\begin{equation}\label{Gram}
  ((e_i, e_j))
  =
  \bpm
  0 & 0 & A \\
  0 & B & 0 \\
  A & 0 & D
  \epm,
\end{equation}
for $A$ and $B$ as in \eqref{Gram0}.
\subsection{$\frac{1}{6}(0,1,1,2)$-extensibility}
Suppose $\Omega(f)$ is defined for $f$ as in \eqref{mfdef} and that $\ec$ is the set of lattice embeddings given in \cite[Lemma 10.10]{DawesKod}.
Unless otherwise stated, we assume that $E \subset L$ is a  primitive totally isotropic sublattice of rank 2 and take the basis of $L$ given in \eqref{Gram}.
\begin{lem}\label{0112-ext-lem}
  If $h$ lies outside of $\opn{EllInt}_{(4-a)k}(\Gamma, \ec)$ then $\Omega(f)$ is $\frac{1}{6}(0,1,1,2)$-extensible for all points in the boundary of $\overline{\fc}(\Gamma)$.
\end{lem}
\begin{proof}
Suppose that $\overline{g} \in G(E)$ acts as $\frac{1}{6}(0,1,1,2)$ at a point in the boundary.
By \cite[Lemma 4.11/Proposition 6.3]{DawesKod}, $\overline{g}$ is represented by $g \in N(E)$ given by
  \begin{align*}
  g =
  \bpm
  U & 0 & W \\
  0 & I & 0 \\
  0 & 0 & Z
  \epm
  && \text{for} &&
  Z =
  \bpm
  a & b \\ c & d
  \epm
  \end{align*}
on the basis \eqref{Gram}.
Suppose $g$ fixes $(u(z)_0, \underline{w}_0, \tau_0)$ the coordinates defined in \S\ref{toroidal_coords}.
The term $(c \tau_0 + d)^{-1}$ is a 6-th root of unity \cite[Proposition 2.28]{GHSK3} and occurs as an eigenvalue of $g$ with multiplicity $\geq 2$.
Therefore, $(c \tau_0 + d)^{-1} \neq 1$ \cite[Proposition 2.28]{GHSK3}\cite[(8.2)]{Kondo} and, by \eqref{gaction}, $g$ fixes some $(0, \underline{0}, \tau_0)$.
As in \cite[Proposition 6.2]{DawesKod}, we introduce new coordinates  $\underline{y}(\underline{x}') = (y_1(\underline{x}'), \ldots, y_4(\underline{x}'))$
on which $g$ acts by $\mathbf{d}_0(g) := \opn{diag}(0, \xi, \xi, \xi^2)$ for $\xi=e^{\pi i /3}$.
These coordinates are given by \cite{Cartan} \cite[p.331]{HuybrechtsK3}
\begin{equation}\label{lincoord}
  \underline{y}(\underline{x'}) = \frac{1}{6} \sum_{i=1}^n \mathbf{d}_0(g)^{-i} g^i(\underline{x}')   
\end{equation}
for  $\underline{x}':=(x_1', x_2', x_3', x_4'):=(u, w_1, w_2, \tau - \tau_0)$.
By \eqref{gaction}, 
\[
g^i: u \mapsto u \opn{exp}_e(T_i(\underline{w}, \tau))
\]
for a function $T_i(\underline{w}, \tau)$.
Hence, by \eqref{lincoord}, $y_1(u, \underline{w}, \tau) = uP(\underline{w}, \tau)$ for some $P(\underline{w}, \tau)$.
As $g$ fixes all points $(y_1, 0, 0, 0)$ on coordinates $\underline{y}$, then $g$ fixes some $(u, \underline{w}, \tau)$ for $u \neq 0$ (i.e. a point of $\dc_L$).
Therefore, $g$ is of finite order with invariant lattice $B$  and $B^{\perp} = U \op U$ \cite[\S4]{DawesKod}.
The lattice $B$ contains a vector $u$ such that there exists an embedding $(u^{\perp} \hookrightarrow L) \in \ec$ \cite[Lemma 10.10]{DawesKod}.
Therefore, the extension criterion of \cite[Proposition 6.3]{DawesKod} is satisfied and the result follows.
\end{proof}
\subsection{Counting $\frac{1}{6}(1,1,1,2)$-points}
We now produce bounds for the number of $\frac{1}{6}(1,1,1,2)$-points in the boundary of $\overline{\fc}(\Gamma)$.
As in \cite{Scattone, DawesFamily}, we start by considering normal forms for \eqref{Gram} to obtain a bound for the $\Gamma$-orbits of totally isotropic rank 2 sublattices in $L$.
We obtain the final bound of Proposition \ref{Pcount_prop} by counting $\frac{1}{6}(1,1,1,2)$-points in each boundary component.
\begin{lem}\label{BDlem}
  Let $E \subset L$ be a primitive totally isotropic sublattice of rank 2 and type $(a,e)$, with $B$ and $D$ as in \eqref{Gram}.
  Then, up to base change in $\opn{Stab}_{\opn{GL}(6, \ZZ)}(E)$, $B$ and $D$ take one of $\#B$ and $\#D$ values, respectively; where   
  \begin{align*}
  \#B \leq
  \frac{\sqrt{\beta}}{\pi} 
  \left ( \opn{log}(4 \beta) + \kappa_{\chi} \right )
  \prod_{p \mid 4 \beta} \left ( 1 - \frac{\psi(p)}{p} \right )
  &&
  \text{and}
  &&
  \#D \leq 4a^3e,
  \end{align*}
  $\beta := \opn{det}L/\opn{det}A^2$, $\psi$ is a Dirichlet character of modulus $D$ and $\kappa_{\chi}$ is as in \eqref{kappadef}.
\end{lem}
\begin{proof}
  As in \cite[(15)]{DawesFamily} (which follows the approach of \cite{Scattone}), we apply the change of basis \eqref{cob} to \eqref{Gram}.
  \begin{equation}\label{cob}
  \bpm
  I & 0 & W \\
  0 & I & 0 \\
  0 & 0 & I
  \epm :
  \bpm
  0 & 0 & A \\
  0 & B & 0 \\
  A & 0 & D
  \epm
  \mapsto
  \bpm
  0 & 0 & A \\
  0 & B & 0 \\
  A & 0 & D + {}^{\intercal}WA + {}^{\intercal}AW
  \epm.
  \end{equation}
  If
  \begin{align*}
    W =
    \bpm
    w & x \\
    y & z 
    \epm
    &&
    \text{then}
    &&
    {}^{\intercal} WA + {}^{\intercal} AW
    =
    a
    \bpm
    2ey & ez+w \\
    ez + w & 2x
    \epm, 
    \end{align*}
  and we can reduce $D_{11}$ modulo $2ae$, $D_{12}$ modulo $a$ and $D_{22}$ modulo $2a$; obtaining a bound for $\#D$ by noting that ${}^{\intercal}D=D$.
  The sublattice defined by $B$ is a negative definite binary quadratic form of discriminant $-4\beta$.
    Therefore, by applying the further change of basis 
    \[
    \bpm
    I & 0 & 0 \\
    0 & U & 0 \\
    0 & 0 & I
    \epm \in \opn{GL}(6, \ZZ), 
    \]
    we can assume that $B$ is one of $h(-4\beta)$ classes, where $h(D)$ is the class number of discriminant $D$.
    We note that  $-4\beta<-4$ as $B$ cannot be unimodular.
    The class number formula \cite[p.49]{Davenport} states that if $D<-4$ then 
    \begin{equation}\label{classno}
    h(D) = \frac{\sqrt{\vert D \vert}}{\pi}L(1, \chi_D), 
    \end{equation}
    for a Dirichlet character $\chi_D$ defined by the Kronecker symbol $\chi_D = \bsm \frac{D}{.} \esm$.
    There exist primitive and principal characters $\psi$ and $\chi_1$, respectively, such that $\chi_D = \psi \chi_1$ \cite[p.170]{Apostol}.
    Hence, by \cite[p.262]{Apostol}, we have 
    \begin{equation}\label{nonprim}
      L(1, \chi_D) = L(1, \psi) \prod_{p \vert D} \left ( 1 - \frac{\psi(p)}{p} \right ).
    \end{equation}
    If $\chi$ is a primitive Dirichlet character and $S$ is a finite set of primes then, by \cite{Louboutin}, 
    \begin{equation}\label{L1bound}
    \left \vert
  L(1, \chi)   \prod_{p \in S} \left (1 - \frac{\chi(p)}{p} \right ) \right \vert
    \leq
    \frac{1}{2}
    \left ( \prod_{p \in S} 1 - \frac{1}{p} \right )
    \left ( \opn{log} q_{\chi} + \kappa_{\chi} + w \opn{log}4 + 2 \sum_{p \in S} \frac{\opn{log} p}{p-1} \right ) + o(1),
    \end{equation}
  where $q_{\chi}$ is the conductor of $\chi$ and $w$ is the number of primes in $S$ not dividing $q_{\chi}$.
  The term $o(1)$ can be calculated explicitly and 
  \begin{equation}\label{kappadef}
  \kappa_{\chi} =
  \begin{cases}
    2 + \gamma_e - \opn{log}4 \pi & \text{if $\chi(-1)=1$} \\
    2 + \gamma_e - \opn{log} \pi & \text{if $\chi(-1)=-1$,}
    \end{cases}
  \end{equation}
  where $\gamma_e$ is Euler's constant.
  Furthermore, if $S = \emptyset$ or $S=\{2 \}$ then $o(1)$ is less than or equal to zero \cite{Louboutin}.
  While the $o(1)$ term in \eqref{L1bound} can be calculated explicitly, we gain little by doing so and assume $S:=\emptyset$.
  As $w=0$ then, by \eqref{nonprim} (noting that $q_{\chi} \leq 4 \beta $), we have
  \[
  \vert L(1, \chi) \vert
  \leq
  \frac{1}{2} 
  \left ( \opn{log}(4 \beta) + \kappa_{\chi} \right )
  \prod_{p \mid 4 \beta} \left ( 1 - \frac{\psi(p)}{p} \right ).
  \]
  Therefore, by \eqref{classno}, 
  \[
  \#B \leq
  \frac{\sqrt{\beta}}{\pi} 
  \left ( \opn{log}(4 \beta) + \kappa_{\chi} \right )
  \prod_{p \mid 4 \beta} \left ( 1 - \frac{\psi(p)}{p} \right ).
  \]
\end{proof}
\begin{lem}\label{indexlem}
  The index $\vert \opn{O}(L) : \Gamma \vert \leq (2880)2^{\nu(2d)}$.
\end{lem}
\begin{proof}
  We use $x \Vert y$ to denote $x \vert y$ and $(x, y/x)=1$. 
  If $2^a \Vert 2d$, $3^b \Vert 2d$ and $d':=2d/(2^a3^b)$ then
  \[
  D(L) = (C_2 \op C_{2^a}) \op (C_3 \op C_{3^b}) \op C_{d'}
  \]
  where the generators of $C_2$, $C_{2^a}$, $C_3$, $C_{3^b}$ are of length $-3/2$, $-2d/2^{2a}$, $-2/3$ and $-2d/3^{2b}$ modulo $2 \ZZ$, respectively.
  By \cite[Lemma 9.3]{DawesKod},
  \[
  \vert \opn{O}(L):\Gamma \vert = 2 \vert \opn{O}(q_L) \vert.
  \]
  As in \cite[Lemma 3.3]{HMvolapps} and \cite{DawesKod}, we estimate $\vert \opn{O}(q_L) \vert$ by counting elements of order $2$, ${2^a}$, $3$ and ${3^b}$ with the above lengths. 
  Suppose $(p,q) \in C_2 \op C_{2^a}$ satisfies  
  \begin{equation}\label{2cong}
    - \frac{3p^2}{2} - \frac{2dq^2}{2^{2a}} \equiv - \frac{2d}{2^{2a}} \bmod{ 2 \ZZ}.
  \end{equation}
  If $p=0$ then \eqref{2cong} is equivalent to $-(2d/2^a) q^2 \equiv (-2d/2^a) \bmod{2^{a+1}}$ or
  \begin{equation}\label{2cong1}
    q^2 - 1 \equiv 0 \bmod{2^{a+1}}.
  \end{equation}
  Congruence \eqref{2cong1} has 1 solution if  $a=1$ and 2 solutions if $a=2$.
  If $a\geq 3$, then at most one of $4 \vert q - 1$ and $4 \vert q+1$ is true, otherwise $q$ is even, implying $q \equiv \pm 1 \bmod 2^a$.
  Therefore, \eqref{2cong1} has at most 2 solutions.
  If $p=1$ then \eqref{2cong} is equivalent to
  \begin{equation}\label{2cong2}
    (2d/2^a)(q^2 - 1) \equiv (3)(2^a) \bmod{2^{a+1}},
  \end{equation}
  which is satisfied if $q^2 - 1 \equiv 0 \bmod{2^{a-1}}$.
  As above, \eqref{2cong2} has at most $4$ solutions;
hence, there are at most $6$ elements in $D(L)$ of order $2^a$ and length $-2d/2^{2a}$ and at most $3$ elements in $D(L)$ of order $2$. 
  Suppose $(p,q) \in C_3 \op C_{3^b}$ of order $3^b$ satisfies 
\begin{equation}\label{3cong}
- \frac{2p^2}{3} - \frac{2dq^2}{3^{2b}} \equiv - \frac{2d}{3^{2b}} \bmod {2 \ZZ}.
\end{equation}
Then \eqref{3cong} has at most 2 solutions if $p=0$ and at most 4 solutions in each of the cases $p=1$ or $2$.
There are at most 8 elements of order 3 in $D(L)$. 
As in \cite[Lemma 10.4]{DawesKod}, there are at most $2^{\nu(d')}$ elements of order $d'$ in $D(L)$ of length $-2d/(2^{2a}3^{2b}) \bmod{ 2 \ZZ}$, hence
\[
\vert \opn{O}(q_L) \vert \leq (1440) 2^{\nu(d')} \leq (1440) 2^{\nu(2d)}
\]
and the result follows.
\end{proof}
\begin{lem}\label{conglem}
  Let $E \subset L$ be a primitive totally isotropic sublattice of rank 2 and type $(a,e)$.
  Suppose the associated boundary component of $\overline{\fc}(\Gamma)$ contains $\frac{1}{6}(1,1,1,2)$-points.
  Then there exists an embedding of the congruence subgroup $\Gamma_1(ae) \hookrightarrow N(E) \cap \opn{O}^+(L)$.
\end{lem}
\begin{proof}
  For convenience, we first define $\Gamma_1(ae) \hookrightarrow \opn{O}^+(L')$ for a sublattice $L' \subset L$ and show it extends to $\opn{O}^+(L)$.
  For the $\ZZ$-basis $\{e_i\}$ of \eqref{Gram} let $L'=\la e_1, e_2, e_4, e_5\ra$.
  Following \cite[p.72]{Scattone} (also used in \cite{DawesFamily}), we show that for each $Z \in \Gamma_1(ae)$ there exists $U, W \in M_2(\ZZ)$ such that 
  \[
  g_Z : =
  \bpm
  U & UW \\
  0 & Z 
  \epm
  \in \opn{O}^+(L').
  \]
  The bilinear form $Q'$ of $L'$ is given by
  \[
  Q' =
  \bpm
  0 & A \\
  A & D
  \epm,
  \]
  implying $g_Z \in \opn{O}^+(L')$ if and only if
  ${}^{\intercal}UAZ = A$ and ${}^{\intercal} WA + {}^{\intercal}AW + {}^{\intercal}ZDZ = D$.
  If $W=(w_{ij})$ and $Z=(z_{ij})$ then
  \begin{equation}\label{WAeq}
    {}^{\intercal}WA + {}^{\intercal}AW =
    \bpm
    2a w_{11} & aw_{12} + aew_{21} \\
    aw_{12} + aew_{21} & 2aew_{22}
    \epm.
    \end{equation}
As $Z \in \Gamma_1(ae)$ then, by direct calculation,   ${}^{\intercal}ZDZ - D \equiv 0 \bmod{ ae}$ and  $U:={}^{\intercal}(AZ^{-1}A^{-1}) \in M_2(\ZZ)$.
Hence, by \eqref{WAeq}, there exists integral $W$ such that  ${}^{\intercal} WA + {}^{\intercal}AW + {}^{\intercal}ZDZ = D$, implying $g_Z \in \opn{O}(L')$.
    By \cite[Lemma 5.7.1]{brieskorn}, $g_Z \in \opn{O}^+(L')$ and from the splitting $L=L' \op B$ given in \eqref{Gram}, $g_Z$ extends to $\opn{O}^+(L)$.
\end{proof}
\begin{lem}\label{Obounds}
For all $\epsilon>0$ there exist constants $K_{\nu, \epsilon}$ and $K_{\sigma_0, \epsilon}$ such that
  $2^{\nu(d)} \leq K_{\nu, \epsilon} d^{\epsilon}$ and
  $\sigma_0(d) \leq K_{\sigma_0, \epsilon} d^{\epsilon}$.
\end{lem}
\begin{proof}
  The bound for $2^{\nu(x)}$ was proved in \cite[Lemma 10.8]{DawesKod}.
  For the case of $\sigma_0(d)$ we use a bound of \cite{NicolasRobin}, which states that, for all $d \geq 2$, 
\[
  \frac{\opn{log}(\sigma_0(d)) \opn{log}(\opn{log}d)}{(\opn{log}(2)\opn{log}(d)} \leq 1.538,
  \]
implying $\sigma_0(d) = O(d^{\epsilon})$.
\end{proof}
\begin{prop}\label{Pcount_prop}
  The number of $\frac{1}{6}(1,1,1,2)$-points in the boundary of $\overline{\fc}(\Gamma)$ is bounded above by
  \begin{align*}
&&  \frac{11520}{\pi}
(\opn{det} L)^{1/2}
(\opn{log}(4 \opn{det} L) + \kappa_{\chi}')
e^{\gamma_e} \opn{log}(\opn{log} (4\opn{det}L))
K_{\nu, \epsilon}^{1+\epsilon} (\opn{det} L)^{3/2 + \epsilon/2} \\
&& \times
\left (
  K_{\sigma_0, \gamma}^2 (\opn{det} L)^{2 \gamma}
  + 4 K_{\sigma_0, \gamma} (\opn{det} L)^{\gamma} + 4
  \right )
  \end{align*}
  where  $\gamma_e$ is Euler's constant and
  $\kappa_{\chi}' = 2 + \gamma_e - \opn{log} \pi$.
\end{prop}
\begin{proof}
  For a subgroup $G \subset \opn{O}^+(L)$, let $N(\overline{\fc}(G))$ be the number of $\frac{1}{6}(1,1,1,2)$-points in the boundary of $\overline{\fc}(G)$.
Suppose $G \trianglelefteq \opn{O}^+(L)$ is of finite index and fix a compactification $\overline{\fc}(\opn{O}^+(L))$.
  As in \cite[Theorem 6.13]{Handbook}, we can assume that $\overline{\fc}(G)$ is defined by the same fans as $\overline{\fc}(\opn{O}^+(L))$, hence $\overline{\fc}(\opn{O}^+(L)) = \overline{\fc}(G)/(\opn{O}^+(L)/G)$.
  By \eqref{GammaDef}, $\widetilde{\opn{SO}}^+(L) \subset \Gamma$ and so  
  \[
  N(\overline{\fc}(\opn{O}^+(L)))
  \leq 
  N(\overline{\fc}(\Gamma))
  \leq 
  N(\overline{\fc}(\widetilde{\opn{SO}}^+(L))).
  \]
  Therefore, we bound $N(\overline{\fc}(\Gamma))$ in terms of $\vert \opn{O}^+(L) : \widetilde{\opn{SO}}^+(L) \vert$ and $N(\overline{\fc}(\opn{O}^+(L)))$. 

  We obtain a bound for $N(\overline{\fc}(\opn{O}^+(L)))$ in terms of the number $\eta(a,e)$ of $\frac{1}{6}(1,1,1,2)$-points above each boundary curve of type $(a,e)$.
As observed in \cite[Proposition 2.27]{GHSK3}, if $P = (0, \underline{0}, \xi)$ is a $\frac{1}{6}(1,1,1,2)$-point on the coordinates $(u(z), \underline{w}, \tau)$ of \S\ref{toroidal_coords}, then $\xi$ is a 3-torsion point of the congruence subgroup $\Gamma_0(e)$.
  There are $\epsilon_3(e)$ 3-torsion points of $\Gamma_3(e)$ up to conjugacy, where 
\begin{equation}\label{epsilon3}
\epsilon_3(\Gamma_0(e)) =
\begin{cases}
  \prod_{p \vert e} \left ( 1 + \left ( \frac{-3}{p} \right ) \right ) & \text{if $9 \not \vert e$} \\
  0 & \text{if $9 \vert e$},
  \end{cases}
\end{equation}
and $\left ( \frac{-3}{p} \right ) = \pm 1$ if and only if $p \equiv \pm 1 \bmod{3}$ \cite[p.96]{Diamond}.
By Lemma \ref{conglem}, $\Gamma_1(ae) \hookrightarrow N(E) \cap \opn{O}^+(L)$, implying 
\[
\eta(a,e) \leq \vert \Gamma_0(e):\Gamma_1(ae) \vert \epsilon_3(e).
\]

Primitive totally isotropic sublattices $E, E' \subset L$ of rank 2 and type $(a,e)$ are equivalent under $\opn{O}(L)$ if they define bases for $L$ with identical Gram matrices in \eqref{Gram}.
Hence, by Lemma \ref{BDlem} (noting that $\vert \opn{O}(L):\opn{O}^+(L) \vert = 2$), we have 
\[
N(\opn{O}^+(L))
\leq
2 \sum_{\substack{\alpha \vert \opn{det} L \\ \alpha < (\opn{det} L)^{1/2}}}
\sum_{\substack{a \vert \alpha \\ a \leq \alpha^{1/2}}}
(\#D) (\# B) \eta(a, \alpha/a^2),
\]
where $\alpha=\opn{det} A$. 
By Lemma \ref{BDlem}, 
\begin{align*}
N(\opn{O}^+(L))
& \leq
8 \sum_{\substack{\alpha \vert \opn{det} L \\ \alpha < (\opn{det} L)^{1/2}}}
\sum_{\substack{a \vert \alpha \\ a \leq \alpha^{1/2}}}
\alpha a \left ( \frac{\sqrt{\opn{det} L}}{\pi \alpha} ( \opn{log}(4 \beta) + \kappa_{\chi}) \prod_{p \vert 4 \beta} \left ( 1 - \frac{\psi(p)}{p} \right) \right ) \eta(a, \alpha/a^2) \\
& \leq
\frac{8 \sqrt{\opn{det} L}}{\pi} \left ( \opn{log}(4 \opn{det} L) + \kappa_{\chi}' \right )
\prod_{p \vert 4 \opn{det} L } \left (1 + \frac{1}{p} \right )
\sum_{\substack{\alpha \vert \opn{det} L \\ \alpha < (\opn{det} L)^{1/2}}}
\sum_{\substack{a \vert \alpha \\ a \leq \alpha^{1/2}}}
a \eta(a, \alpha/a^2)
,
\end{align*}
where $\kappa_{\chi}'$ is the maximum value of $\kappa_{\chi}$ defined in \eqref{kappadef}.

Using the bounds \cite{Diamond}
\begin{align*}
  \vert \opn{SL}(2, \ZZ): \Gamma_0(n)\vert  =  n \prod_{p \vert n} \left (1 + \frac{1}{p} \right )  && \text{and} && \vert \opn{SL}(2, \ZZ): \Gamma_1(n) \vert  =  n^2 \prod_{p \vert n} \left (1 - \frac{1}{p^2} \right ), 
\end{align*}
we conclude that $\eta(a, \alpha/a^2) \leq \alpha \epsilon_3(e)$.
Hence, by \eqref{epsilon3} and Lemma \ref{Obounds}, 
\begin{align}
\sum_{\substack{\alpha \vert \opn{det} L \\ \alpha < (\opn{det} L)^{1/2}}}
\sum_{\substack{a \vert \alpha \\ a \leq \alpha^{1/2}}}
a  \eta(a, \alpha/a^2)
& \leq
K_{\nu, \epsilon} (\opn{det} L)^{\epsilon/2}
\sum_{\substack{\alpha \vert \opn{det} L \\ \alpha < (\opn{det} L)^{1/2}}}
\sum_{\substack{a \vert \alpha \\ a \leq \alpha^{1/2}}} \frac{\alpha a}{a^{2 \epsilon}} \nonumber \\
 & \leq K_{\nu, \epsilon} (\opn{det} L)^{\epsilon/2} \sum_{\substack{\alpha \vert \opn{det} L \\ \alpha < (\opn{det} L)^{1/2}}} \alpha
\sum_{\substack{a \vert \alpha \\ a \leq \alpha^{1/2}}} \frac{a}{a^{2 \epsilon}} \label{nusum1} \\
& \leq K_{\nu, \epsilon} (\opn{det} L)^{\epsilon/2}
\sum_{\substack{\alpha \vert \opn{det} L \\ \alpha < (\opn{det} L)^{1/2}}}
\left ( \frac{\alpha^{3/2}}{2 \alpha^{\epsilon}} \right) \left ( K_{\sigma_0, \gamma} \alpha^{\gamma} + 2\right), \label{nusum2}
\end{align}
where we obtained \eqref{nusum2} by bounding for the inner sum in \eqref{nusum1},  noting it has at most $(\sigma_0(a)/2 + 1)$ terms each $\leq \alpha^{1/2}/\alpha^{\epsilon}$ and using Lemma \ref{Obounds}.
By repeating the argument for the outer sum,  \eqref{nusum2} is bounded above by
\begin{align*}
& \frac{K_{\nu, \epsilon}}{4} (\opn{det} L)^{3/2 - \epsilon/2} \left (
  K_{\sigma_0, \gamma}^2 (\opn{det} L)^{2 \gamma} + 4 K_{\sigma_0, \gamma} (\opn{det} L)^{\gamma} + 4 \right ).
\end{align*}
Therefore,
\begin{align*}
  N(\opn{O}^+(L))
  \leq  & \frac{2}{\pi}
(\opn{det} L)^{1/2}
(\opn{log}(4 \opn{det} L) + \kappa_{\chi}')
  \left ( \prod_{p \vert 4 \opn{det}L}  \left ( 1 + \frac{1}{p} \right ) \right) \\
  & \times K_{\nu, \epsilon} (\opn{det} L)^{3/2 - \epsilon/2} \left (
  K_{\sigma_0, \gamma}^2 (\opn{det} L)^{2 \gamma}
  + 4 K_{\sigma_0, \gamma} (\opn{det} L)^{\gamma} + 4
  \right ).
\end{align*}
If $\opn{det} L \geq 8$ then, by \cite{SolePlanat}, we have
\[
\prod_{p \vert 4 \opn{det}L} \left ( 1 + \frac{1}{p} \right )
\leq
e^{\gamma_e} \opn{log}(\opn{log} (4\opn{det}L)) 
\]
where $\gamma_e$ is Euler's constant.
Hence, by Lemma \ref{indexlem},
\begin{align*}
  E(\widetilde{\opn{O}}^+(L))
  \leq   & \frac{5760}{\pi}
(\opn{det} L)^{1/2}
(\opn{log}(4 \opn{det} L) + \kappa_{\chi}')
e^{\gamma_e} \opn{log}(\opn{log} (4\opn{det}L))
K_{\nu, \epsilon}^{1+\epsilon} (\opn{det} L)^{3/2 + \epsilon/2} \\
& \times
\left (
  K_{\sigma_0, \gamma}^2 (\opn{det} L)^{2 \gamma}
  + 4 K_{\sigma_0, \gamma} (\opn{det} L)^{\gamma} + 4
  \right ).
\end{align*}
By definition of $\Gamma$ in \eqref{GammaDef}, $\vert \widetilde{\opn{SO}}^+(L): \Gamma \vert = 2$ and the result follows.
\end{proof}
 \section{Obstruction Calculations}\label{obstruction}
We now define $\opn{EllCusp}_{(4-a)k}(\Gamma)$ and show that $\fc(\Gamma)$ is of general type when $d \gg 0$.
\begin{deflem}\label{EllCuspDefLem}
  Let $f=f_a^kh$ be the modular form in \eqref{mfdef} and let $\opn{EllCusp}_{(4-a)k}(\Gamma) \subset M_{(4-a)k}(\Gamma)$ be the subspace defined in \eqref{EllCuspDef}.
If $h$ lies outside of $\opn{RefObs}_{(4-a)k}(\Gamma)$, $\opn{EllInt}_{(4-a)k}(\Gamma, \ec)$ and $\opn{EllCusp}_{(4-a)k}(\Gamma)$ then $\Omega(f)$ defines a pluricanonical form on $\widetilde{\fc}(\Gamma)$. 
\end{deflem}
\begin{proof}
For a neat normal subgroup $\Gamma' \trianglelefteq \Gamma$ \cite{AMRT}, let $G:=\Gamma/\Gamma'$.
The assignment $f \mapsto \Omega(f)$ of \S\ref{LWCFT_sec} defines a map
\[
\Omega: f_a^k M_{(4-a)k}(\Gamma)
\rightarrow H^0(\overline{\fc}(\Gamma'), kK)^G.
\]
If $P \in \overline{\fc}(\Gamma)$ is a $g$-point with tangent space $T_P \overline{\fc}(\Gamma')$ then, by restriction, we obtain
\[
H^0(\overline{\fc}(\Gamma'), kK)^G \rightarrow H^0(T_P \overline{\fc}(\Gamma'), kK)^g
\]
for any $g \in G_P$.
Sections in $H^0(T_P \overline{\fc}(\Gamma'), kK)^g$  descend to pluricanonical forms on the regular part of $X_g$ and, for a resolution $\widetilde{X}_g \rightarrow X_g$, sections in $H^0(T_P \overline{\fc}(\Gamma'), kK)^g$ pull back to $H^0(\widetilde{X}_g, \lc)$ for some line bundle $\lc$.
Under the composition 
\[
f_a^k M_{(4-a)k}(\Gamma)
\rightarrow
H^0(\widetilde{X}_g, \lc),
\]
$\Omega(f)$ defines a pluricanonical form on $\widetilde{\fc}(\Gamma)$ above $P \in \overline{\fc}(\Gamma)$ if it defines a pluricanonical form on $\widetilde{X}_g$ for all $g \in G_P$ \cite[Proposition 3.2]{Tai}.
If $f = f_a^k h$
and $h$ lies outside of $\opn{RefObs}_{(4-a)k}(\Gamma)$ and $\opn{EllInt}_{(4-a)k}(\Gamma, \ec)$, then 
this condition is satisfied if $h$ lies outside $\opn{EllCusp}_{(4-a)k}(\Gamma)$, where
\begin{equation}\label{EllCuspDef}
  \opn{EllCusp}_{(4-a)k}(\Gamma) \cong
  \bigoplus_{\substack{\text{$\frac{1}{6}(1,1,1,2)$ points} \\ P \in \overline{\fc}(\Gamma) }} H^0(\widetilde{X}_g, \lc)/H^0(\widetilde{X}_g, kK).
\end{equation}
\end{proof}
Let $Y_F$ be the toric variety defined by the fan $F$ \cite{Fulton}.
As in \cite{DawesKod} (following \cite[p.434]{Tai}), the singularity $\frac{1}{6}(1,1,1,2)$ can be realised as a toric variety $Y_{\sigma}$.
Let $N$ be a lattice with $\ZZ$-basis $\{e_1, \ldots, e_4 \}$  and let  $N \subset \overline{N}$ be the overlattice defined by $\overline{N} = \la N, \frac{1}{6}(1,1,1,2) \ra \subset N \otimes \QQ$.
  If $\sigma$ is the cone $\sigma = \la e_1, \ldots, e_4 \ra \subset \overline{N} \otimes \RR$, then $\frac{1}{6}(1,1,1,2) = Y_{\sigma}$ \cite[\S6]{DawesKod}.
  There exists a resolution
  \begin{equation}\label{resolution}
    Y_{\Sigma} \rightarrow Y_{\sigma},
  \end{equation}
  obtained by subdividing $\sigma$ to obtain the fan $\Sigma$ with rays $v_1, \ldots, v_6 \in \overline{N}$ given by 
  $v_1 = (1,0,0,0)$,
  $v_2 = (0,1,0,0)$,
  $v_3 = (0,0,1,0)$,
  $v_4 = (0,0,0,1)$,
  $v_5 = \frac{1}{2}(1,1,1,0)$ and 
  $v_6 = \frac{1}{6}(1,1,1,2)$
  on the $\QQ$-basis $\{e_i\}$;
  and maximal cones 
  $\la v_1, v_2, v_4, v_6 \ra$,
  $\la v_1, v_2, v_5, v_6 \ra$,
  $\la v_1, v_3, v_4, v_6 \ra$,
  $\la v_1, v_3, v_5, v_6 \ra$,
  $\la v_2, v_3, v_4, v_6 \ra$ and
  $\la v_2, v_3, v_5, v_6 \ra$ \cite[\S6]{DawesKod}.
  We will use $E_i$ to denote the divisor of $Y_{\Sigma}$ defined by the ray $v_i$ \cite{Fulton}.
  As usual, for an arbitrary toric variety, we let $K$ denote the canonical bundle.
  \begin{lem}\label{bundlelem}
    Suppose $g \in \opn{GL}(4, \CC)$ acts as $\frac{1}{6}(1,1,1,2)$.
    If $6 \vert k$ and  $\omega \in H^0(\CC^4, kK)^g$ is a $g$-invariant pluricanonical then $\omega$ defines a section of $H^0(Y_{\Sigma}, \lc)$, where
    \[
    \lc = kK - \frac{k}{2}E_5 + \frac{k}{6}E_6.
    \]
  \end{lem}
  \begin{proof}
    We calculate $\lc$ by working locally.
  Let $\{z_1^*, \ldots, z_4^* \}$ be a chart for $Y_{\Sigma}$, defined by the rays of a maximal cone $\sigma^* \subset \Sigma$.
  If $\eta$ is the restriction of $\Omega(f)$ to the cover $\CC^4$ of $X_g \cong Y_{\sigma}$ then, as in \cite{Tai},   
  \begin{align*}
    \eta  
    & = g(z_1, \ldots, z_4) \frac{(dz_1 \wedge \ldots \wedge dz_4)^{\otimes k}}{(z_1 \ldots z_4)^k} \\
    & = g(z_1^*, \ldots, z_4^*) \frac{(dz_1^* \wedge \ldots \wedge dz_4^*)^{\otimes k}}{(z_1^* \ldots z_4^*)^k}. 
  \end{align*}
  Hence, if $z_j$ is defined by the ray $v = \sum \lambda_i e_i$ then \cite{Fulton, Tai}
  \[
  \opn{ord}_{z_j^*} g(z_1^*, \ldots, z_4^*) = \sum \lambda_i \opn{ord}_{z_i}g(z_1, \ldots, z_4) 
  \]
and $\eta$ is a section of 
  \[
  \lc:=  kK - \frac{k}{2} E_5 + \frac{k}{6} E_6.
  \]
  \end{proof}
 The support of $E_i$ in $Y_{\Sigma}$ is the orbit closure $V(v_i)$, which can be realised as the toric variety defined by the fan $\opn{Star}_{\Sigma}(v_i)$ in the lattice $\overline{\overline{N}}:=\overline{N}/\la v_i \ra$ \cite{Fulton}.
 In particular, the rays of $\opn{Star}_{\Sigma}(v_6)$ are given by
 $\overline{v}_1 = (-1,-1,-2)$,
 $\overline{v}_2 = (1,0,0)$,
 $\overline{v}_3 = (0,1,0)$,
 $\overline{v}_4 = (0,0,1)$ and 
 $\overline{v}_5 = (0,0,-1)$.
 We let $\overline{E}_j$ denote the divisor in $\opn{Star}_{\Sigma}(v_6)$ defined by the ray $\overline{v}_i$.
 \begin{lem}\label{movsuplem}
 By moving the support of $K$ and $E_6$ in $Y_{\Sigma}$, we obtain divisors $K' \sim K$ and $E_6' \sim E_6$, both missing $E_6$, where
 $K' = 5E_1 - E_2 - E_3 - E_4 + 2E_5$
 and
 $E_6' = -6E_1 - 3E_5$.
 If $Z_1:=K'\vert_{E_6}$ and $Z_2:=E_6'\vert_{E_6}$, then 
\begin{align*}
  Z_1 = 5\overline{E_1} - \overline{E_2} - \overline{E_3} - \overline{E_4} + 2 \overline{E_5}
  &&
\text{and}
&&
  Z_2 = -6\overline{E}_1 - 3\overline{E}_5.
  \end{align*}
  \end{lem}
  \begin{proof}
Fix a basis
 $f_1:=v_6$,
 $f_2:=v_2$,
 $f_3:=v_3$ and
 $f_4:=v_4$ for $\overline{N}$.
 For each cone $\delta \in \Sigma$, suppose the Cartier divisor $k'(\delta) \in \overline{N}^{\vee}$ represents $K'$ on the chart $U_{\delta}$ defined by $\delta$ \cite[p.66]{Fulton}.
 Let $\overline{\delta}$ denote the image of $\delta$ in $\opn{Star}_{\Sigma}(v_6)$ 
 and let $\overline{k'}(\delta)$ denote the image of $k'(\delta)$ in $\overline{\overline{N}}^{\vee}$.
We calculate, 
\begin{center}
    \begin{tabular}{|l|l|l|l|}
      \hline
      & & $k'(\delta)$ & $\overline{k}'(\delta)$ \\
      \hline
      
      & $\{v_1, v_2, v_4, v_6 \}$, & $(0,-1,-2,-1)$, & $(-1,-2,-1)$ \\
     & $\{v_1, v_2, v_5, v_6 \}$, & $ (0,-1,0,-2)$, & $ (-1,0,-2)$ \\
 $\delta$      &$ \{v_1, v_3, v_4, v_6 \}$, & $ (0,-2,-1,-1)$, & $ (-2,-1,-1)$ \\
   &   $\{v_1, v_3, v_5, v_6 \}$, & $ (0, 0,-1,-2)$, & $ (0,-1,-2)$ \\
 & $ \{v_2, v_3, v_4, v_6 \}$, & $(0,-1,-1,-1)$, & $ (-1,-1,-1)$ \\
      &    $ \{v_2, v_3, v_5, v_6 \}$, & $ (0,-1,-1,-2)$, & $ (-1,-1,-2)$. \\
      \hline
    \end{tabular}
    \end{center}
    Therefore,
  \begin{align*}
  Z_1:=K'\vert_{E_6} = 5\overline{E_1} - \overline{E_2} - \overline{E_3} - \overline{E_4} + 2 \overline{E_5} 
  && \text{and} &&
  Z_2:=E_6'\vert_{E_6} = -6\overline{E}_1 - 3\overline{E}_5.
  \end{align*}
  \end{proof}
\begin{lem}\label{P_obs}
  If $X_g = \frac{1}{6}(1,1,1,2)$ and $\widetilde{X}_g := Y_{\Sigma}$ as in \eqref{resolution} then,
  \begin{equation}\label{obstruction_space}
    \opn{dim} 
    H^0(\widetilde{X}_g, \lc)/H^0(\widetilde{X}_g, kK)
    \leq 
    \frac{1}{288}(k^4 + 34k^3 + 430k^2 + 2400).
  \end{equation}
\end{lem}
\begin{proof}
We obtain the bound \eqref{obstruction_space} by considering the simpler bundle
 \[
 \lc(a) := kK' + aE_6'
 \]
 on $Y_{\Sigma}$. 
 By twisting $\lc(a)$ with  
  \[
  0 \rightarrow \oc(-E_6) \rightarrow \oc \rightarrow \oc \vert_{E_6} \rightarrow 0
  \]
  we obtain the long exact sequence (c.f. \cite{OGabs} and  \cite[Proposition 4.3]{HSabs})
  \[
  \ldots
  \rightarrow H^0(Y_{\Sigma}, \lc(a-1))
  \rightarrow H^0(Y_{\Sigma}, \lc(a))
  \rightarrow H^0(E_6, \lc(a)\vert_{E_6})
  \rightarrow \ldots.
  \]
  Hence, 
  \begin{align*}
    h^0 (Y_{\Sigma}, kK - \frac{k}{2} E_5 + \frac{k}{6} E_6) - h^0(Y_{\Sigma}, kK)
    & \leq
    h^0 (Y_{\Sigma}, kK + \frac{k}{6} E_6) - h^0(Y_{\Sigma}, kK) \\
    & \leq 
    \sum_{a=1}^{k/6} h^0(E_6, \lc(a)\vert_{E_6}).
  \end{align*}
If $P_D$ denotes the polyhedron defined by a torus-invariant divisor $D$, then \cite[p.67]{Fulton} 
  \[
  P_{\lc(a)\vert E_6} \subset kP_{Z_1} + a P_{Z_2}. 
  \]
  The polyhedron  $P_{Z_2}$ is empty and $6P_{Z_1}$ has Ehrhart polynomial
  \[
  L(6P_{Z_1}, k) := 18k^3 + \frac{45}{2}k^2 + \frac{17}{2} k + 1
  \]
  (as can be calculated using the  computer algebra system \ttt{Sage}, for example).
  Therefore,
  \begin{align*}
    \sum_{a=1}^{k/6} h^0(E_6, \lc(a) \vert_{E_6})
    & \leq \sum_{a=1}^{k/6} L(6P_{Z_1}, k) \\
    & \leq \frac{1}{288}(k^4 + 34k^3 + 430k^2 + 2400) 
  \end{align*}
  and the result follows.
\end{proof}
\begin{prop}\label{beta_prop}
  For $\epsilon > 0 $, we have
  \[
\alpha_2(a) \leq   \frac{256 D(a)}{225 \pi^2} \frac{K_{\nu, \epsilon}^3 K_{\sigma_0, \eta}}{3+\epsilon} 3^{\epsilon} 8^{3+\epsilon} d^{3/2 + 3 \epsilon/2 + \nu}
  \]
  where
  \begin{equation}\label{Da}
    D(a) = 2(4-a)^3 + 6(4-a)^2 + 8(4-a) + 4.
  \end{equation}
\end{prop}
\begin{proof}
  We proceed as in \cite[Proposition 10.12]{DawesKod}.
  As in \cite{DawesKod}, define $\alpha_2(a)$ by
\[
  \opn{dim} \bigoplus_{K \in \ec} \bigoplus_{i=0}^{2k-1}
M_{k(4-a)+i}(\Gamma \cap \opn{O}(K))
= \alpha_2(a)k^4 + O(k^3), 
\]
where the set $\ec$ consists of lattice embeddings $2U \op \la -2r \ra \hookrightarrow L$ for $0<r \leq 2 \sqrt{d}$\cite[Proposition 6.2/Lemma 10.10]{DawesKod}.
Suppose 
\[
\alpha_2(a)=\sum_{K \in \ec/\Gamma} \beta_K(a),
\]
where $\beta_K(a)$ denotes the contribution to $\alpha_2(a)$ from a fixed embedding $K \in \ec/\Gamma$.
(We may also use $\beta_K$ to denote $\beta_K(a)$ where no confusion is likely to arise.) 
We begin by bounding the number of embeddings $K \hookrightarrow L$ up to $\Gamma$-equivalence.
  By \cite[p.129-130]{Nikulin}, for fixed $K = 2U \op \la -2r \ra$, there are at most 4 $\Gamma$-equivalence classes of embeddings $K \hookrightarrow L$ for each tuple $(H_K, H_L, \gamma; T, \gamma_T)$, where
  $H_K \subset D(L)$, $H_L \subset D(L)$ are subgroups;
  $\gamma:q_K \vert H_K \xrightarrow{\sim} q_L \vert H_L$ is an isomorphism of groups preserving the restriction of discriminant forms $q_K$ and $q_L$ \cite{Nikulin};
  $T$ is a negative definite lattice of rank 1 (with discriminant form given in \cite[(76)]{DawesKod}) and $\gamma_T \in \opn{Aut}(q_T)$.
  As in \cite{DawesKod}, because $D(K)$ is cyclic then $H_K$ is uniquely determined by $s:=\vert H_K \vert$ and there are at most 12 embeddings of $H_L \subset D(L)$.
  Hence, as in \cite[(78)]{DawesKod}, we have 
  \[
  \alpha_2(a) \leq 48 \sum_{\substack{s \vert 12d \\
    s \leq 4 \sqrt{d}}} \sum_{\substack{K \cong 2U \op \la -2r \ra \\ 0<r \leq 2 \sqrt{d}}} \sum_{\gamma, T, \gamma_T} \beta_K.
  \]
  There are at most $\vert \opn{O}(H_L) \vert$ isomorphisms $\gamma:H_K \xrightarrow{\sim} H_L$.  
As $s \vert 2r$ and from the bound $\vert \opn{O}(H_L) \vert \leq s$, we have \[
  \alpha_2(a) \leq 48
  \sum_{\substack{s \vert 12d \\ s \leq 4 \sqrt{d}}}
  s
  \left (
  \sum_{\substack{i=1 \\ K \cong 2U \la -si \ra \\ 2 \vert si}}^{(4 \sqrt{d})/s}
  \sum_{T, \gamma_T} \beta_K 
  \right ).
  \]
  The lattice $T$ is uniquely determined by $(D(K), D(L), \gamma)$ and there are at most $2^{\nu(6dr)}$ choices for $\gamma_T$ for a given $K$ \cite{DawesKod}, implying 
  \[
  \alpha_2(a) \leq 48
  \sum_{\substack{s \vert 12d \\ s \leq 4 \sqrt{d}}}
  s
  \sum_{\substack{i=1 \\ K \cong 2U \la -si \ra \\ 2 \vert si}}^{(4 \sqrt{d})/s}
  2^{\nu(6dsi)} \beta_K.
  \]
  As in \cite[Proposition 10.12]{DawesKod}, for $K=2U \op \la -2r \ra$, we have
  \[
  \beta_K
  \leq
  \frac{16 D(a) r^2}{675 \pi^2}
  \]
  with $D(a)$ as in \eqref{Da}.
  Therefore,
  \begin{align}
    \alpha_2(a)
    & \leq
    \frac{256 D(a)}{ 225 \pi^2} \sum_{\substack{s \vert 12d \\ s \leq 4\sqrt{d}}} s \sum_{\substack{i=1 \\ K \cong 2U \op \la -si \ra \\ 2 \vert si}}^{(4 \sqrt{d})/s} (si)^2 2^{\nu(3d)} 2^{\nu(si)} \nonumber \\
    & \leq
    \frac{256 D(a)}{225 \pi^2} K_{\nu, \epsilon}^3 (3d)^{\epsilon} \sum_{\substack{s \vert 12 d \\s \leq 4 \sqrt{d}}} \sum_{i=1}^{(4 \sqrt{d})/s} s^{\epsilon+3} i^{2 + \epsilon} \nonumber \\
    & \leq
    \frac{256 D(a)}{225 \pi^2} K_{\nu, \epsilon}^3 (3d)^{\epsilon} \sum_{\substack{s \vert 12 d \\ s \leq 4 \sqrt{d}}} s^{3 + \epsilon} \sum_{i=1}^{(4 \sqrt{d})/s} i^{2 + \epsilon} \label{ellsumbound} \\
\intertext{and, by comparing the innermost sum in \eqref{ellsumbound} to an integral over $i \in [0, (8\sqrt{d})/s]$, we obtain }
    & \leq
    \frac{256 D(a)}{225 \pi^2} \frac{(8\sqrt{d})^{3 + \epsilon}}{3+\epsilon} (3d)^{\epsilon} K_{\nu, \epsilon}^3 \sum_{\substack{s \vert 12 d \\ s \leq 4 \sqrt{d}}} 1 \nonumber \\
    & \leq
    \frac{256 D(a)}{225 \pi^2} \frac{K_{\nu, \epsilon}^3 K_{\sigma_0, \eta}}{3+\epsilon} 3^{\epsilon} 8^{3+\epsilon} d^{3/2 + 3 \epsilon/2 + \nu}. \nonumber
  \end{align}
\end{proof}
\begin{thm}\label{gtthm}
  The orthogonal modular varieties $\fc(\Gamma)$, and hence associated moduli spaces of deformation generalised Kummer varieties considered in \cite{DawesKod}, are of general type if $d \gg 0$.
\end{thm}
\begin{proof}
By \cite[Theorem Corollary 7.1]{DawesKod}, $S_3(\Gamma) \neq 0$ for $d \gg 0$.
  In \cite[Lemma 9.3]{DawesKod}, we show that
  \begin{align*}
  \alpha_0 \geq \frac{(12d)^{5/2}}{497664 \zeta(3) \pi^3}  
  && \text{and} && 
  \alpha_1
  \leq
  \frac{108199}{360 \pi^2} d^2 + \frac{3243945}{64 \pi^2} d^{1/6} 
  \end{align*}
  for $d \gg 0$.
  We obtained bounds for $\alpha_2$ in Proposition \ref{beta_prop} and, 
  by Propositions \ref{Pcount_prop} and \ref{P_obs}, we have
  \begin{align*}
    && \alpha_3 \leq \frac{40}{\pi}
    (\opn{det} L)^{1/2}
    (\opn{log}(4 \opn{det} L) + \kappa_{\chi}')
    e^{\gamma_e} \opn{log}(\opn{log} (4\opn{det}L))
    K_{\nu, \epsilon}^{1+\epsilon} (\opn{det} L)^{3/2 + \epsilon/2} \\
    && \times
    \left (
  K_{\sigma_0, \gamma}^2 (\opn{det} L)^{2 \gamma}
  + 4 K_{\sigma_0, \gamma} (\opn{det} L)^{\gamma} + 4
  \right ).
  \end{align*}
The result then follows by considering \eqref{alphabound} for $a=3$, with the statement about moduli spaces following as in \cite[Corollary 11.1]{DawesKod}, by applying \cite[Theorem 3.10]{Handbook}.
\end{proof}
 \bibliographystyle{alpha}
\bibliography{bib}{}

\begin{thebibliography}{AMRT10}

\bibitem[AMRT10]{AMRT}
A.~Ash, D.~Mumford, M.~Rapoport, and Y-S. Tai.
\newblock {\em Smooth compactifications of locally symmetric varieties}.
\newblock Cambridge Mathematical Library. Cambridge University Press,
  Cambridge, second edition, 2010.

\bibitem[Apo98]{Apostol}
T.M. Apostol.
\newblock {\em Introduction to analytic number theory}.
\newblock Undergraduate Texts Math. New York, NY: Springer, 1998.

\bibitem[BJ06]{BorelJi}
A.~Borel and L.~Ji.
\newblock {\em Compactifications of symmetric and locally symmetric spaces}.
\newblock Mathematics: Theory \& Applications. Birkh{\"a}user Boston, Inc.,
  Boston, MA, 2006.

\bibitem[Bri83]{brieskorn}
E.~Brieskorn.
\newblock {\em Die {M}ilnorgitter der exzeptionellen unimodularen
  {S}ingularit{\"a}ten}.
\newblock Bonner Mathematische Schriften [Bonn Mathematical Publications], 150.
  Universit{\"a}t Bonn, Mathematisches Institut, Bonn, 1983.

\bibitem[Car57]{Cartan}
H.~Cartan.
\newblock Quotient d'un espace analytique par un groupe d'automorphismes.
\newblock In {\em Algebraic geometry and topology}, pages 90--102. Princeton
  University Press, Princeton, N. J., 1957.

\bibitem[CS99]{SPLAG}
J.H. Conway and N.J.A. Sloane.
\newblock {\em Sphere packings, lattices and groups}, volume 290 of {\em
  Grundlehren der Mathematischen Wissenschaften}.
\newblock Springer-Verlag, New York, third edition, 1999.

\bibitem[Dav13]{Davenport}
H.~Davenport.
\newblock {\em Multiplicative number theory}, volume~74.
\newblock Springer, 2013.

\bibitem[Daw23]{DawesFamily}
M.~Dawes.
\newblock The {Baily}-{Borel} compactification of a family of orthogonal
  modular varieties.
\newblock {\em Osaka J. Math.}, 60(2):333--350, 2023.

\bibitem[Daw24]{DawesKod}
M.R. Dawes.
\newblock {On the Kodaira dimension of the moduli of deformation generalised
  Kummer varieties}.
\newblock {\em To appear in Algebraic Geometry, arXiv preprint
  arXiv:1710.1672}, 2024.

\bibitem[DS05]{Diamond}
F.~Diamond and J.~Shurman.
\newblock {\em A first course in modular forms}, volume 228 of {\em Graduate
  Texts in Mathematics}.
\newblock Springer-Verlag, New York, 2005.

\bibitem[Ful93]{Fulton}
W.~Fulton.
\newblock {\em Introduction to toric varieties}, volume 131 of {\em Annals of
  Mathematics Studies}.
\newblock Princeton University Press, Princeton, NJ, 1993.

\bibitem[GH95]{GritsenkoHulekIrrational}
V.~Gritsenko and K.~Hulek.
\newblock Appendix to the paper ``{Irrationality} of the moduli spaces of
  polarized abelian surfaces''.
\newblock In {\em Abelian varieties. Proceedings of the international
  conference held in Egloffstein, Germany, October 3-8, 1993}, pages 83--84.
  Berlin: Walter de Gruyter, 1995.

\bibitem[GHS07a]{HMvolapps}
V.A. Gritsenko, K.~Hulek, and G.K. Sankaran.
\newblock The {H}irzebruch-{M}umford volume for the orthogonal group and
  applications.
\newblock {\em Doc. Math.}, 12:215--241, 2007.

\bibitem[GHS07b]{GHSK3}
V.A. Gritsenko, K.~Hulek, and G.K. Sankaran.
\newblock The {K}odaira dimension of the moduli of {$K3$} surfaces.
\newblock {\em Invent. Math.}, 169(3):519--567, 2007.

\bibitem[GHS08]{HMvol}
V.A. Gritsenko, K.~Hulek, and G.K. Sankaran.
\newblock Hirzebruch-{M}umford proportionality and locally symmetric varieties
  of orthogonal type.
\newblock {\em Doc. Math.}, 13:1--19, 2008.

\bibitem[GHS13]{Handbook}
V.A. Gritsenko, K.~Hulek, and G.K. Sankaran.
\newblock Moduli of {K}3 surfaces and irreducible symplectic manifolds.
\newblock In {\em Handbook of moduli. {V}ol. {I}}, volume~24 of {\em Adv. Lect.
  Math. (ALM)}, pages 459--526. Int. Press, Somerville, MA, 2013.

\bibitem[Gri94]{GritsenkoAbelianK3}
V.~Gritsenko.
\newblock Modular forms and moduli spaces of abelian and {{\(K3\)}} surfaces.
\newblock {\em St. Petersbg. Math. J.}, 6(6):65--102, 1994.

\bibitem[GS96]{GritsenkoSankaranAbelian}
V.A. Gritsenko and G.K. Sankaran.
\newblock Moduli of abelian surfaces with a {{\((1,p^2)\)}} polarization.
\newblock {\em Izv. Math.}, 60(5):893--900, 1996.

\bibitem[HS94]{HSabs}
K.~Hulek and G.K. Sankaran.
\newblock {The Kodaira dimension of certain moduli spaces of abelian surfaces}.
\newblock {\em Compositio Mathematica}, 90(1):1--35, 1994.

\bibitem[{Huy}16]{HuybrechtsK3}
D.~{Huybrechts}.
\newblock {\em {Lectures on {K3} surfaces}}, volume 158.
\newblock Cambridge: Cambridge University Press, 2016.

\bibitem[Kne02]{kneser}
M.~Kneser.
\newblock {\em {Quadratische Formen. Neu bearbeitet und herausgegeben in
  Zusammenarbeit mit Rudolf Scharlau.}}
\newblock Berlin: Springer, 2002.

\bibitem[Kon93]{Kondo}
S.~Kond{\=o}.
\newblock On the {K}odaira dimension of the moduli space of {$K3$} surfaces.
\newblock {\em Compositio Math.}, 89(3):251--299, 1993.

\bibitem[Lou04]{Louboutin}
S.R. Louboutin.
\newblock Explicit upper bounds for {{\(| L(1,\chi)|\)}} for primitive
  characters {{\(\chi\)}}.
\newblock {\em Q. J. Math.}, 55(1):57--68, 2004.

\bibitem[{Ma}24]{MaIrregularCusps}
S.~{Ma}.
\newblock {Irregular cusps of orthogonal modular varieties}.
\newblock {\em Can. J. Math.}, 2024.

\bibitem[{Nik}80]{Nikulin}
V.V. {Nikulin}.
\newblock {Integral symmetric bilinear forms and some of their applications}.
\newblock {\em {Math. USSR, Izv.}}, 14:103--167, 1980.

\bibitem[NR83]{NicolasRobin}
J.-L. {Nicolas} and G.~{Robin}.
\newblock {Majorations explicites pour le nombre de diviseurs de \(n\).}
\newblock {\em {Can. Math. Bull.}}, 26:485--492, 1983.

\bibitem[O'G89]{OGabs}
K.G. O'Grady.
\newblock On the {K}odaira dimension of moduli spaces of abelian surfaces.
\newblock {\em Compositio Math.}, 72(2):121--163, 1989.

\bibitem[{Pia}69]{Piatetski-Shapiro}
I.I. {Piatetski-Shapiro}.
\newblock {Automorphic functions and the geometry of classical domains.}
\newblock {Mathematics and its Applications. 8. New York-London-Paris: Gordon
  and Breach Science Publishers. viii, 264 p. (1969).}, 1969.

\bibitem[Rei87]{YPG}
M.~Reid.
\newblock Young person's guide to canonical singularities.
\newblock In {\em Algebraic geometry, {B}owdoin, 1985 ({B}runswick, {M}aine,
  1985)}, volume~46 of {\em Proc. Sympos. Pure Math.}, pages 345--414. Amer.
  Math. Soc., Providence, RI, 1987.

\bibitem[{Sat}80]{Satake1}
I.~{Satake}.
\newblock {Algebraic structures of symmetric domains.}
\newblock {Kano Memorial Lectures 4. Publications of the Mathematical Society
  of Japan, 14. Tokyo: Iwanami Shoten, Publishers. Princeton, New Jersey:
  Princeton University Press.}, 1980.

\bibitem[Sca87]{Scattone}
F.~Scattone.
\newblock On the compactification of moduli spaces for algebraic {$K3$}
  surfaces.
\newblock {\em Mem. Amer. Math. Soc.}, 70(374), 1987.

\bibitem[SP11]{SolePlanat}
P.~Sol{\'e} and M.~Planat.
\newblock Extreme values of the {Dedekind} {{\(\Psi\)}} function.
\newblock {\em J. Comb. Number Theory}, 3(1):33--38, 2011.

\bibitem[Tai82]{Tai}
Y-S. Tai.
\newblock On the {K}odaira dimension of the moduli space of abelian varieties.
\newblock {\em Invent. Math.}, 68(3):425--439, 1982.

\end{thebibliography}
\noindent \texttt{matthew.r.dawes@bath.edu}
\end{document}